\documentclass[12pt,a4paper,reqno]{amsart}
\usepackage[utf8]{inputenc}
\usepackage[T1]{fontenc}
\usepackage{mathptmx}
\usepackage{amsfonts}

\usepackage[english]{babel}
\usepackage{amsmath, amsthm, amsfonts,amssymb}
\usepackage{hyperref}
\newtheorem{thm}{Theorem}[section]
\newtheorem{cor}[thm]{Corollary}
\newtheorem{lem}[thm]{Lemma}
\newtheorem{prop}[thm]{Proposition}
\newtheorem{rem}[thm]{Remark}

\theoremstyle{plain}

\theoremstyle{definition}

\newtheorem{defn}[thm]{Definition}
\theoremstyle{remark}

\newcommand\blfootnote[1]{%
	\begingroup
	\renewcommand\thefootnote{}\footnote{#1}%
	\addtocounter{footnote}{-1}%
	\endgroup
}

\usepackage{vmargin}

\setpapersize{A4}
\setmargins{2.5cm}       
{1.5cm}                        
{16.5cm}                      
{23.42cm}                    
{10pt}                           
{1cm}                           
{0pt}                             
{2cm}                           

\newcommand{\N}{\mathbb{N}}

\newcommand{\X}{\mathbb{X}}

\def\ba{\begin{eqnarray*}}
	\def\ea{\end{eqnarray*}}

\def\bee{\begin{equation}}

\def\ene{\end{equation}}

\def\e{\varepsilon}

\newcommand{\sgn}{\mathrm{sgn \,}}
\newcommand{\supp}{\mathrm{supp \,}}
\title{Characterization of weight-semi-greedy bases}
\date{}
\begin{document}
	\author{P. M. Bern\'a}
	\address{Pablo M. Bern\'a
		\\
		Departmento de Matem\'aticas
		\\
		Universidad Aut\'onoma de Madrid
		\\
		28049 Madrid, Spain} \email{pablo.berna@uam.es}
	
	\maketitle

	\begin{abstract} One classical result in greedy approximation theory is that almost-greedy and semi-greedy bases are equivalent in the context of Schauder bases in Banach spaces with finite cotype. This result was proved by S. J. Dilworth, N. J. Kalton and D. Kutzarova in \cite{DKK} and, recently, the first author in \cite{B}  proved that the condition of finite cotype can be removed in this result.  In \cite{DKTW}, the authors extend the notion of semi-greediness to the context of weights and proved the following: if $w$ is a weight and $\mathcal B$ is a Schauder basis in a Banach space $\mathbb X$ with finite cotype, then $w$-semi-greediness and $w$-almost-greediness are equivalent notions. In this paper, we prove the same characterization but removing the condition of finite cotype and, also, we try to relax the condition of Schauder in the characterization of semi-greediness using the $\rho$-admissibility, notion introduced recently in \cite{BBGHO2}.\end{abstract}
	\blfootnote{\hspace{-0.031\textwidth} 2000 Mathematics Subject Classification. 46B15, 41A65.\newline
		The first author was supported by a PhD fellowship of the program "Ayudas para contratos predoctorales para la formación de doctores 2017" (MINECO, Spain) and the grants MTM-2016-76566-P (MINECO, Spain) and 19368/PI/14 (\emph{Fundaci\'on S\'eneca}, Regi\'on de Murcia, Spain).  }
	
	\begin{section}{INTRODUCTION}
		Let $(\X,\Vert \cdot \Vert)$ be a Banach space over $\mathbb F$ ($\mathbb F$ denotes the real field $\mathbb R$ or the complex field $\mathbb C$) and let  $\mathcal{B}=(e_n)_{n=1}^{\infty}$ be a semi-normalized \textit{Markushevich basis} of $\X$ with biorthogonal functionals $(e_n^{*})_{n=1}^{\infty}$, that is:
	\begin{itemize}
		\item[a)] $0<c_1:=\inf_n\lbrace \Vert e_n\Vert, \Vert e_n^*\Vert\rbrace \leq \sup_n\lbrace \Vert e_n\Vert, \Vert e_n^*\Vert\rbrace=:c_2<\infty$.
		\item[b)] $e_j^*(e_i)=\delta_{j,i}$.
		\item[c)] $\mathbb X = \overline{span\lbrace e_i : i\in\mathbb N\rbrace}$.
		\item[d)] If $e_j^*(x)=0$ for all $j\in\mathbb N$, then $x=0$.
	\end{itemize}
	
	We say that $\mathcal B$ is a semi-normalized \textit{strong Markushevich basis} if a)-d) are satisfied and
	\begin{itemize}
		\item[e)] $\overline{span\lbrace e_i : i \in A\rbrace}=\lbrace x\in\mathbb X : e_j^*(x)=0\; \forall j\not\in A\rbrace$.
	\end{itemize}
Throughout the paper, we will refer to a semi-normalized strong Markushevich basis $\mathcal B$ as a basis. 
Also, we will say that $\mathcal B$ is a \textit{Schauder basis} if $\mathcal B$ is a basis in the above sense and if
	\begin{itemize}
		\item[f)]  $\mathfrak{K}_b:=\sup_m \Vert P_m\Vert<\infty$, where $P_m(\sum_j a_j e_j) = \sum_{j=1}^{m}a_j e_j$ is the $m$-th partial sum.
	\end{itemize}
	
	As usual  $\supp(x)=\{n\in \N: e_n^*(x)\ne 0\}$, given a finite set $A\subset \mathbb N$, $|A|$ denotes the cardinality of the set $A$, 
	$$\mathbb N^m = \lbrace A\subset\mathbb N : \vert A\vert =m\rbrace,\; \mathbb N^{<\infty}=\cup_{m=0}^\infty \mathbb N^m,$$
	$P_A$ is the projection operator, that is, $P_A(\sum_j a_j e_j)=\sum_{j\in A} a_j e_j$, $P_{A^c}=\text{I}_{\mathbb X}-P_A$, $\mathbf{1}_{\varepsilon A}=\sum_{n\in A} \varepsilon_ne_n$ where $\varepsilon=(\varepsilon_n)_{n}$ is a sign, that is, $\vert \varepsilon_n\vert = 1$ (where $\varepsilon_n$ could be real or complex), $\mathbf{1}_A=\sum_{n\in A}e_n$ and, for $A, B$ finite sets, $A<B$ means that $\max_{i\in A} i < \min_{j\in B} j$.  
	
	In the year 1999 (\cite{KT}), S. V. Konyagin and V. N. Temlyakov introduced the \textit{Thresholding Greedy Algorithm} (TGA): given a basis $\mathcal B$ in a Banach space and $x \sim \sum_{i=1}^{\infty}e_i^{*}(x)e_i \in \X$, the collection $(G_m(x))_{m=1}^\infty$ is a \textit{greedy approximation of x}, where $G_m(x)=\sum_{n\in \Lambda}e_n^*(x)e_n$, and the set $\Lambda$ is any set of cardinality $m$ satisfying the following condition:
	$$\min_{n\in \Lambda}\vert e_n^*(x)\vert \geq \max_{n\not\in \Lambda}\vert e_n^*(x)\vert.$$
	The set $\Lambda$ is called a \textit{greedy set}.
	
	In general, $(G_m(x))_m$ can not be unique since we can have some coefficients with the modulus. Hence, we consider the natural ordering existing in $\mathbb N$ to solve this fact.  Define the \textit{natural greedy ordering} for $x$ as the map $\rho: \mathbb{N}\longrightarrow\mathbb{N}$ such that $\supp(x) \subset \rho(\mathbb{N})$ and so that if $j<k$ then either $\vert e_{\rho(j)}^*(x)\vert > \vert e_{\rho(k)}^*(x)\vert$ or $\vert e_{\rho(j)}^*(x)\vert = \vert e_{\rho(k)}^*(x)\vert$ and $\rho(j)<\rho(k)$. The $m$-\textit{th greedy sum} of $x$ is $$\mathcal{G}_m[\mathcal B, \mathbb X](x)=\mathcal{G}_m(x) = \sum_{j=1}^m e_{\rho(j)}^*(x)e_{\rho(j)},$$
	and the sequence of maps $(\mathcal{G}_m)_{m=1}^\infty$ is known as the \textit{Thresholding Greedy Algorithm} associated to $\mathcal{B}$ in $\X$. Of course, we can write $\mathcal G_m(x) = \sum_{k \in A_m(x)} e_k^*(x) e_k$, where $A_m(x) = \{ \rho(n) : n \leq m\}$
	is the \textit{greedy set} of $x$ with cardinality $m$: $\min_{k \in A_m(x)} \vert e_k^*(x) \vert \geq \max_{k \notin A_m(x)} \vert e_k^*(x) \vert$. 
	
	The terminology of the Thresholding Greedy Algorithm can be found, for instance, in \cite{DKKT,DKT2,KT, Woj}. In \cite{KT}, the authors defined \textit{quasi-greedy bases}:
	
	\begin{defn}
		We say that $\mathcal B$ in a Banach space $\mathbb X$ is \textbf{quasi-greedy} if there exists a positive constant $C$ such that
		\begin{eqnarray}\label{def:quasi}
		\Vert x-\mathcal G_m(x)\Vert\leq C\Vert x\Vert,\; \forall x\in\mathbb X, \forall m\in\mathbb N.
		\end{eqnarray}
		The least constant that verifies \eqref{def:quasi} is denoted by $C_q$ and we say that $\mathcal B$ is $C_q$-quasi-greedy.
	\end{defn}
 In \cite{Woj}, P. Wojtaszczyk proved that a basis $\mathcal B$ is quasi-greedy if and only if $$\lim_{m\rightarrow+\infty}\Vert x-\mathcal G_m(x)\Vert = 0,\; \forall x\in\mathbb X.$$
 Then, from the point of view of the approximation, quasi-greediness is the minimum condition that guarantees the convergence of the TGA, but there are others greedy-type bases that we need to attack the problem that we want to study. On the one hand, we have \textit{greedy bases}, notion introduced by S. V. Konyagin and V. N. Temlyakov in \cite{KT}. We say that $\mathcal B$ is \textit{greedy} if there exists a positive constant $C$ such that
	$$\Vert x-\mathcal G_m(x)\Vert \leq C\sigma_m(x),\; \forall x\in\mathbb X, \forall m\in\mathbb N,$$
	where
	$$\sigma_m(x,\mathcal B)_{\mathbb X}=\sigma_m(x):=\inf\left\lbrace \left\Vert x-\sum_{n\in A}a_n e_n\right\Vert : a_n\in\mathbb F, A\subset\mathbb N, \vert A\vert\leq m\right\rbrace.$$
%
%
	In this paper we focus our attention on almost-greedy bases, notion introduced by S. J. Dilworth, N. J. Kalton, D. Kutzarova and V. N. Temlyakov in \cite{DKKT}. We say that $\mathcal B$ is \textit{almost-greedy} if there exists a positive constant $C$ such that
	$$\Vert x-\mathcal G_m(x)\Vert\leq C\inf\left\lbrace\left\Vert x-P_A(x)\right\Vert : \vert A\vert= m\right\rbrace.$$
	
 Later on, in \cite{AA}, the authors proved that the notion of almost-greediness is equivalent to $$\Vert x-\mathcal G_m(x)\Vert \leq C\tilde{\sigma}_m(x),\; \forall m\in\mathbb N, \forall x\in\mathbb X,$$ where $C$ is the same constant than in the definition of almost-greediness and
 $$\tilde{\sigma}_m(x,\mathcal B)_{\mathbb X}=\tilde{\sigma}_m(x):=\inf\left\lbrace\Vert x-P_A(x)\Vert : A\subset\mathbb N, \vert A\vert\leq m\right\rbrace.$$
 	In \cite{DKKT}, the authors proved that a basis is almost-greedy if and only if the basis is quasi-greedy and democratic (that is, there exists a positive constant $C$ such that $\Vert \mathbf{1}_A\Vert \leq C\Vert\mathbf{1}_B\Vert$, for any $A, B\in\mathbb N^{<\infty}$ and $\vert A\vert\leq\vert B\vert$). 
%
%
	Thanks to a work of G. Kerkyacharian, D. Picard and V. N. Temlyakov (\cite{GPT}) motivated by the work of A. Cohen, R. A. DeVore and R. Hochmuth (\cite{CDH}), we consider a generalization of  almost-greedy bases. We define a weight $w$ as any collection $w=(w_i)_{i=1}^\infty\in (0,\infty)^{\mathbb N}$. If $A\subset \mathbb N$, $w(A)=\sum_{i\in A}w_i$ denote the $w$-measure of $A$. We define the error 
	$$\tilde{\sigma}_\delta^w(x,\mathcal B)_{\mathbb X}=\tilde{\sigma}_\delta^w(x):=\inf\lbrace \Vert x-P_A(x)\Vert : A\in\mathbb N^{<\infty}, w(A)\leq\delta\rbrace.$$
	
	\begin{defn}[\cite{DKTW}]
		We say that $\mathcal B$ is \textbf{$w$-almost-greedy} if there exists a positive constant $C$ such that
		\begin{eqnarray}\label{walgreedy}
		\Vert x-\mathcal{G}_m(x)\Vert \leq C\tilde{\sigma}_{w(A_m(x))}^w(x),\; \forall x\in\mathbb X, \forall m\in\mathbb N.
		\end{eqnarray}
		We denote by $C_{a}$ the least constant that verifies \eqref{walgreedy} and we say that $\mathcal B$ is $C_a$-$w$-almost-greedy.
	\end{defn}

	\begin{defn}
		We say that $\mathcal B$ is \textbf{$w$-super-democratic} if there exists a positive constant $C$ such that
		\begin{eqnarray}\label{superd}
		\Vert \mathbf{1}_{\varepsilon A}\Vert\leq C \Vert\mathbf{1}_{\varepsilon' B}\Vert,
		\end{eqnarray}
		for any pair of sets $A, B\in\mathbb N^{<\infty}$ with $w(A)\leq w(B)$ and for all pair of signs $\varepsilon,\varepsilon'$.
		We denote by $C_{s}$ the least constant that verifies \eqref{superd} and we say that $\mathcal B$ is $C_{s}$-$w$-super-democratic.
		
		If in \eqref{superd} we add the condition $A\cap B=\emptyset$, we say that $\mathcal B$ is \textbf{$C_{sd}$-$w$-disjoint-super-democratic}.
		If $\varepsilon\equiv\varepsilon'\equiv 1$ in \eqref{superd}, we say that $\mathcal B$ is \textbf{$C_d$-$w$-democratic}.
	\end{defn}
	In \cite{BDKOW}, the authors gave the following characterization of $w$-almost-greedy bases.
	\begin{thm}[\cite{BDKOW}]\label{walm} Assume that $\mathcal B$ is a basis in a Banach space $\mathbb X$.
		\begin{itemize}
			\item If $\mathcal B$ is $C_q$-quasi-greedy and $C_d$-$w$-democratic, then $\mathcal B$ is $C_a$-$w$-almost-greedy with $$C_a\leq 8C_q^4C_d+C_q+1.$$
			\item If $\mathcal B$ is $C_a$-$w$-almost-greedy, then $\mathcal B$ is $C_q$-quasi-greedy and $C_d$-$w$-democratic with $$\max\lbrace C_q, C_d\rbrace\leq C_a.$$
		\end{itemize}
	\end{thm}
	
	Here, we reformulate the above theorem using the $w$-disjoint-super-democracy.
	
	\begin{thm}\label{main1}
		Assume that $\mathcal B$ is a basis in a Banach space $\X$. 
		\begin{itemize}
			\item[a)] If $\mathcal B$ is $C_a$-$w$-almost-greedy, then $\mathcal B$ is $C_q$-quasi-greedy, $C_{sd}$-$w$-disjoint-super-democratic and $C_s$-$w$-super-democratic, with $$\max\lbrace C_q, C_{sd}\rbrace \leq C_a,\;\; C_s \leq 2\kappa C_a,$$ where $\kappa = 1$ if $\mathbb F=\mathbb R$ and $\kappa=2$ if $\mathbb F=\mathbb C$.
			\item[b)] If $\mathcal B$ is $C_q$-quasi-greedy and $C_{sd}$-$w$-disjoint-super-democratic, then the basis is $C_a$-$w$-almost-greedy with $$C_a\leq C_q+2C_qC_{sd}.$$
		\end{itemize}
	\end{thm}

\begin{rem}\label{rem1}
Theorem \ref{main1} shows that $C_a=O(C_qC_{sd})$ instead of $C_a = O(C_q^4C_d)$ as Theorem \ref{walm} shows. This is an improvement respect to the order of the constants as we can see using the Proposition \ref{proprem} of Section \ref{final}.
\end{rem}
	
	In \cite{DKK}, S. J. Dilworth, N. J. Kalton and D. Kutzarova study an equivalence of almost-greedy bases from a new point of view to improve the rate of convergence. For this equivalence, the authors introduced the notion of \textit{semi-greedy bases}. Let $A_m(x)$ the greedy set of $x$ of cardinality $m$. Define the \textit{$m$-th Chebyshev-greedy sum} as any element $\mathcal{CG}_m(x)\in span\lbrace e_i : i\in A_m(x)\rbrace$ such that 
	$$\Vert x-\mathcal{CG}_m(x)\Vert = \min\left\lbrace \left\Vert x-\sum_{n\in A_m(x)}a_n e_n\right\Vert : a_n\in\mathbb F\right\rbrace.$$
	The collection $\lbrace \mathcal{CG}_m\rbrace_{m=1}^\infty$ is the \textit{Thresholding Chebyshev Greedy Algorithm} (TCGA). A basis $\mathcal B$ is \textit{semi-greedy} if there exists a positive constant $C$ such that
	\begin{eqnarray*}\label{semig}
		\Vert x-\mathcal{CG}_m(x)\Vert \leq C\sigma_m(x),\; \forall x\in\mathbb X, \forall m\in\mathbb N.
	\end{eqnarray*}
	The first authors that studied the relation between semi-greediness and almost-greediness were S. J. Dilworth, N. J. Kalton and D. Kutzarova in\cite{DKK}.
	\begin{thm}{\cite[Theorem 3.6]{DKK}}
		Assume that $\mathcal B$ is a Schauder basis in a Banach space $\mathbb X$ with finite cotype. Then, $\mathcal B$ is semi-greedy if and only if $\mathcal B$ is almost-greedy.
	\end{thm}	
	
	Recently, the first author proved in \cite{B} that the condition of finite cotype can be removed.
	
	\begin{thm}{\cite[Theorem 1.10]{B}}
		A Schauder basis $\mathcal B$ in a Banach space $\mathbb X$ is semi-greedy if and only if $\mathcal B$ is almost-greedy.
	\end{thm}
	
	Focusing our attention in the weighted case, in \cite{DKTW}, the authors extend the definition of semi-greediness.
	
	\begin{defn}[\cite{DKTW}]
		We say that $\mathcal B$ is \textbf{$w$-semi-greedy} if there exists a positive constant $C$ such that
		\begin{eqnarray}\label{wsemigreedy}
		\Vert x-\mathcal{CG}_m(x)\Vert \leq C\sigma_{w(A_m(x))}^w(x),\; \forall x\in\mathbb X, \forall m\in\mathbb N,
		\end{eqnarray}
		where
		$$\sigma_\delta^w(x,\mathcal B)_{\mathbb X}=\sigma_\delta^w(x):=\inf\left\lbrace \left\Vert x-\sum_{n\in A}a_n e_n\right\Vert : A\in\mathbb N^{<\infty}, w(A)\leq\delta, a_n\in\mathbb F\right\rbrace.$$
		
		We denote by $C_{sg}$ the least constant that verifies \eqref{wsemigreedy} and we say that $\mathcal B$ is $C_{sg}$-$w$-semi-greedy.
	\end{defn}
	
	Here, we study the equivalence between $w$-semi-greediness and $w$-almost-greediness removing the condition of finite cotype following the spirit of \cite{B} in the world of Schauder bases in Banach spaces.
	\begin{rem}
	For $w=(1,1,...)$, that is, $w(A)=\vert A\vert$, we recover the definitions of super-democracy, disjoint-super-democracy, almost-greediness and semi-greediness. 
	\end{rem}
	\begin{thm}\label{main}
		Assume that $\mathcal B$ is a Schauder basis in a Banach space.
		\begin{itemize}
			\item[a)] If $\mathcal B$ is $C_{sg}$-$w$-semi-greedy, then $\mathcal B$ is $C_q$-quasi-greedy and $C_s$-$w$-super-democratic with
			$$C_q\leq C_{sg}\mathfrak{K}_b(1+(1+\mathfrak{K}_b)C_{sg}+c_2^2),\;\; C_s\leq \mathfrak{K}_bC_{sg}((1+\mathfrak{K}_b)C_{sg}+c_2^2).$$
			\item[b)] If $\mathcal B$ is $C_q$-quasi-greedy and $C_{s}$-$w$-super-democratic, then $\mathcal B$ is $C_{sg}$-$w$-semi-greedy with $$C_{sg}\leq C_q+4C_qC_s.$$
		\item[c)] If $\mathcal B$ is $C_q$-quasi-greedy and $C_{sd}$-$w$-disjoint-super-democratic, then $\mathcal B$ is $C_{sg}$-$w$-semi-greedy with $$C_{sg}\leq C_q+4C_q^2C_{sd}.$$
		\end{itemize}
	\end{thm}
\begin{rem}
	In \cite{DKTW}, the authors proved that $C_q$-quasi-greediness and $C_d$-$w$-democracy implies $C_{sg}$-$w$-semi-greediness with $C_{sg}=(C_q^3C_d)$. Our Theorem \ref{main} shows that we can get $C_{sg}=O(C_q^2C_{sd})$ or $C_{sg}=O(C_qC_s)$. As in the Remark \ref{rem1}, the last bound is an improvement respect to the bound $C_{sg}=(C_q^3C_d)$ using Proposition \ref{proprem}.
\end{rem}
	Hence, for $w=(1,1,...)$, we recover the result proved in \cite{B} as we say in the following corollary.
	\begin{cor}\label{corsemi}
	Assume that $\mathcal B$ is a Schauder basis in a Banach space. The following are equivalent:
	\begin{itemize}
	\item[a)] $\mathcal B$ is semi-greedy.
	\item[b)] $\mathcal B$ is quasi-greedy and super-democratic.
	\item[c)] $\mathcal B$ is quasi-greedy and disjoint-super-democratic.
\end{itemize}
	\end{cor}

	
	The structure of the paper is the following: in Section \ref{pre}, we write and study some preliminary results that we will use in the proof of the main results. In Section \ref{propc}, we improve a recent result proved in \cite{BDKOW} establishing the relation between $w$-semi-greediness and the so called Property (C). In Section \ref{proofs}, we prove Theorems \ref{main1} and \ref{main}. In Section \ref{rho}, we relax the condition of Schauder bases in the characterization of semi-greediness and in Section \ref{final}, we show the relation between $w$-super-democracy and $w$-disjoint-super-democracy for any weight $w$ and we add some open questions.

\section{Preliminary results}\label{pre}

To prove the main theorems of this paper, we need the followings results concerning convexity, the truncation operator and some properties of weights.

\begin{lem}{\cite[Lemma 2.7]{BBG}}\label{conv}
	For every finite set $A\subset \mathbb N$, we have
	$$\text{co}(\lbrace \mathbf{1}_{\varepsilon A} : \vert \varepsilon\vert=1\rbrace)= \left\lbrace \sum_{n\in A}z_n e_n : \vert z_n\vert \leq 1\right\rbrace,$$
	where $\text{co}(S)=\lbrace \sum_{j=1}^n \alpha_j x_j : x_j \in S, 0\leq \alpha_j\leq 1, \sum_{j=1}^n \alpha_j = 1, n\in\mathbb N\rbrace$.
\end{lem}

For each $\alpha>0$, we define the \textit{truncation function} of $z\in\mathbb F$ as
$$T_\alpha (z) = \alpha\sgn(z),\; \vert z\vert > \alpha,\;\; T_\alpha(z) = z,\; \vert z\vert \leq\alpha.$$
We can extend $T_\alpha$ to an operator in $\X$ by $T_\alpha(x)\sim \sum_{i=1}^\infty T_\alpha(e_i^*(x))e_i$, that is,
$$T_\alpha(x):=\sum_{i=1}^\infty T_\alpha(e_i^*(x))e_i = \alpha\mathbf{1}_{\varepsilon\Gamma_\alpha}+P_{\Gamma_\alpha^c}(x),$$
where $\Gamma_\alpha = \lbrace n : \vert e_n^*(x)\vert > \alpha\rbrace$ and $\varepsilon_j = \sgn(e_j^*(x))$ with $j\in \Gamma_\alpha$.
Hence, this is a well-defined operator for all $x\in \X$ since $\Gamma_\alpha$ is a finite set.

This operator was introduced in \cite{DKK} to show the equivalence between almost-greediness and semi-greediness and they proved that for quasi-greedy bases, this operator is uniformly bounded. Also, in \cite{BBG}, the authors showed the same result but with a slight improvement of the boundedness constant.

\begin{prop}{\cite[Lemma 2.5]{BBG}}\label{truncation}
	Assume that $\mathcal B$ is $C_q$-quasi-greedy basis in a Banach space $\X$. Then, for every $\alpha>0$,
	$$\Vert T_\alpha(x)\Vert\leq C_q\Vert x\Vert,\; \forall x\in\X.$$
\end{prop}

The last one result that we will use is related to weights. 
	\begin{prop}\label{p:find c0}
	Let $\mathcal{B}$ be a basis in a Banach space $\mathbb X$. 
	\begin{enumerate}
		\item[i)] Assume that $w(A)\leq \lim\sup_{n\rightarrow\infty} w_n$. If $\mathcal B$ is $C_{sg}$-$w$-semi-greedy, then
		$$\max_{\vert\varepsilon\vert=1}\Vert \mathbf{1}_{\varepsilon A}\Vert \leq C_{sg}c_2(1+c_2^2).$$
		If in addition $\mathcal B$ is Schauder, it is possible to get that $\max_{\vert\varepsilon\vert=1}\Vert \mathbf{1}_{\varepsilon A}\Vert \leq 2C_{sg}\mathfrak{K}_b$.
	
		If $\mathcal B$ is $C_{sd}$-$w$-disjoint-super-democratic, then
		$$\max_{\vert\varepsilon\vert=1}\Vert \mathbf{1}_{\varepsilon A}\Vert \leq c_2C_{sd}.$$
		\item[ii)] If $\mathcal B$ is $C_{sg}$-$w$-semi-greedy or $C_{sd}$-$w$-disjoint-super-democratic and $\sup_n w_n = \infty$ or $\sum_n w_n <\infty$, then $\mathcal B$ is equivalent to the $c_0$-basis.
	\item[iii)] If $\mathcal B$ is $C_{sg}$-$w$-semi-greedy or $C_{sd}$-$w$-disjoint-super-democratic and $\inf_n w_n = 0$, $\mathcal B$ contains a subsequence equivalent to the $c_0$-basis.
	\end{enumerate}
\end{prop}

\begin{proof}
	The case of $w$-disjoint-super-democracy is proved in \cite{BDKOW} assuming $w$-super-democracy, but the same proof is also valid for $w$-disjoint-super-democracy. The case of $w$-semi-greediness is proved in \cite{DKTW} but assuming that the basis is Schauder. Here, we show this result for general bases using similar ideas.
	\begin{itemize}
		\item[i)] Find $n\in\mathbb N\setminus A$ such that $w(A)<w_n$. Hence, if we consider the element $x:=\mathbf 1_{\varepsilon A} + (1+\delta)e_n$ with $\delta>0$ and $\e$ a sign, applying the TCGA with $\lbrace n\rbrace$ the greedy set of $x$, 
		\begin{eqnarray*}
			\Vert \mathbf{1}_{\varepsilon A}\Vert &\leq& \Vert \mathbf{1}_{\varepsilon A}+\lambda e_n\Vert + \Vert \lambda e_n\Vert\\
			&\leq& C_{sg}\sigma_{w_n}^w(x) + \lambda c_2\\
			&\leq& C_{sg}\Vert (1+\delta)e_n\Vert + c_2\vert e_n^*(\mathbf{1}_{\varepsilon A}+\lambda e_n)\vert\\
			&\leq& C_{sg}c_2(1+\delta) +c_2^2\Vert  \mathbf{1}_{\varepsilon A}+\lambda e_n\Vert \\
			&\leq& (1+\delta)(C_{sg}c_2+C_{sg}c_2^3).
		\end{eqnarray*}
	Taking limits when $\delta$ goes to $0$, $\Vert\mathbf{1}_{\varepsilon A}\Vert \leq C_sgc_2(1+c_2^2)$.
		\item[ii)] If $\sup_n w_n<\infty$, by a), $\Vert \mathbf{1}_{\varepsilon A}\Vert \leq C_{sg}c_2(1+c_2^2)$ for any finite set $A$ and for all possible sign $\varepsilon$. Hence, the basis is equivalent to the $c_0$-basis. 
	If $\sum_n w_n<\infty$, we can choose a number $m\in\mathbb N$ such that $\sum_{i=m+1}^\infty w_n<w_1$. We can assume that $\min_{i\in A} i\geq m+1$. Hence, with the same procedure as in a), $\Vert \mathbf{1}_{\varepsilon A}\Vert \leq C_{sg}c_2+C_{sg}c_2^3$.
		\item[iii)] Choosing a subsequence $(n_k)_k$ such that $\sum_{k=1}^\infty w_{n_k}<\infty$, we apply the item b) and we have that $(e_{n_k})_k$ is equivalent to the $c_0$-basis.
	\end{itemize}
\end{proof}

%
\section{The Property (C)}\label{propc}
It is well known that one of the properties that quasi-greediness preserves is the so called Property (C). 

\begin{defn}
We say that a basis $\mathcal B$ in a Banach space $\mathbb X$ has the \textbf{Property (C)} if for any $x\in\mathbb X$ and $G$ a greedy set of $x$, there exists a positive constant $C$ such that
$$\min_{j\in G}\vert e_j^*(x)\vert \Vert \mathbf{1}_{\varepsilon G}\Vert \leq C\Vert x\Vert,\; \forall \vert\varepsilon\vert=1.$$
We denote by $C_u$ the least constant that verifies the above inequality and we say that $\mathcal B$ has the $C_u$-Property (C).
\end{defn}

Although quasi-greediness implies Property (C), the converse is false as \cite[Example 5.5]{BBG} shows. The following is a known inequality from \cite{DKKT}.
\begin{lem}{\cite[Lemma 2.2]{DKKT}}\label{propC}
	If $\mathcal B$ is a $C_q$-quasi-greedy basis in $\mathbb X$, then, for all $x\in\mathbb X$ and for all greedy set $G$ of $x$, we have 
	\begin{eqnarray}\label{ineqC}
	\min_{j\in G}\vert e_j^*(x)\vert \Vert \mathbf{1}_{\varepsilon G}\Vert \leq 2C_q\Vert x\Vert,
	\end{eqnarray}
	where $\varepsilon = \lbrace\sgn(e_j^*(x))\rbrace$. 
\end{lem}
\begin{rem}
Since any quasi-greedy basis is unconditional for constant coefficients (see \cite{Woj}), that is, $\Vert \mathbf{1}_{\varepsilon A}\Vert \approx \Vert \mathbf{1}_A\Vert$ for any sign $\varepsilon$ and any finite set $A$, from Lemma \ref{propC} we can deduce that any quasi-greedy basis has the Property (C).
\end{rem}

\cite[Proposition 4.10]{BDKOW} shows that any Schauder and $w$-semi-greedy basis satisfies the Property (C) assuming that $0<\inf_n w_n\leq \sup_n w_n<\infty$. Here, we extend this result proving that $w$-semi-greediness implies the Property (C) for any weight $w$. Before that, we study the following lemma.

\begin{lem}\label{lemma1}
	Assume that $\mathcal B$ is a Schauder basis with constant $\mathfrak{K}_b$ and $C_{sg}$-$w$-semi-greedy in a Banach space $\mathbb X$. Then, if $F$ is a set such that $F>\supp(x)$ and $w(F)\leq w(G)$ for some greedy set $G$, hence
	$$\min_{i\in G}\vert e_i^*(x)\vert \Vert\mathbf{1}_{\eta F}\Vert \leq C_{sg}(1+\mathfrak{K}_b)\Vert x\Vert,\; \forall \vert\eta\vert=1.$$ 
\end{lem}
\begin{proof}
Take $x\in\mathbb X$, $G$ a greedy set of $x$ and a set $F$ and a sign $\eta$ as in the statement of the lemma. Define the following element $y:= \min_{i\in G}\vert e_i^*(x)\vert \mathbf{1}_{\eta F}+ P_{G^c}(x)+ \sum_{i\in G}(e_i^*(x)+\delta\varepsilon_i)e_i$, where $\delta>0$ and $\varepsilon\equiv \lbrace\sgn (e_j^*(x))\rbrace$. Then, for the element $y$, the set $G$ is a greedy set. Hence, applying the TCGA,
\begin{eqnarray*}
\min_{i\in G}\vert e_i^*(x)\vert \Vert\mathbf{1}_{\eta F}\Vert&\leq& (1+\mathfrak{K}_b)\Vert \min_{i\in G}\vert e_i^*(x)\vert \mathbf{1}_{\eta F}+ P_{G^c}(x)+ \sum_{i\in G}a_ie_i\Vert\\
&\leq& (1+\mathfrak{K}_b)C_{sg}\sigma_{w(G)}^w(y)\\
&\leq& (1+\mathfrak{K}_b)C_{sg}\Vert y-\min_{i\in G}\vert e_i^*(x)\vert \mathbf{1}_{\eta F}\Vert\\
&=& (1+\mathfrak{K}_b)\Vert P_{G^c}(x)+ \sum_{i\in G}(e_i^*(x)+\delta\varepsilon_i)e_i\Vert.
\end{eqnarray*}
Taking limits when $\delta$ goes to $0$, we obtain the result.
\end{proof}

\begin{prop}\label{proop}
Let $\mathcal B$ be a Schauder basis with constant $\mathfrak{K}_b$ in a Banach space $\mathbb X$. If $\mathcal B$ is $C_{sg}$-$w$-semi-greedy, then $\mathcal B$ has the $C_u$-Property (C) with $C_u\leq \mathfrak{K}_bC_{sg}((1+\mathfrak{K}_b)C_{sg}+c_2^2)$.
\end{prop}
\begin{proof}
	Take $x\in\mathbb X$ and let $G$ be a greedy set of $x$, $\alpha=\min_{i\in G}\vert e_i^*(x)\vert$ and $\vert \varepsilon\vert=1$. We consider different cases (these cases are inspired by \cite{DKTW}).

\item \textbf{Case 1:} $\sum_{n=1}^\infty w_n = \infty$ and $\sup_n w_n <\infty$.

\textbf{Case 1.1:} If $w(G)>\lim\sup_{n\rightarrow\infty} w_n$, since $\sum_{n}w_n = \infty$, we can choose $E$ and $n_0\in\mathbb N$ with $E> \supp(x)$ and $n_0>\max E$ such that $$w (E)\leq w(G)<w(E)+w_{n_0}.$$ Define then the element $y:=\alpha \mathbf{1}_{\varepsilon G}+(\alpha+\delta)\mathbf{1}_{F}$, where $\delta>0$ and $F=E\cup\lbrace n_0\rbrace$. Then, a greedy set of $y$ is $F$ and hence, applying the TCGA,
	\begin{eqnarray*}
		\alpha\Vert\mathbf{1}_{\varepsilon G}\Vert &\leq& \mathfrak{K}_b\Vert\alpha \mathbf{1}_{\varepsilon G}+\sum_{i\in F}a_i e_i\Vert\\\nonumber
		&\leq& \mathfrak{K}_bC_{sg}\sigma_{w(F)}^w(y)\leq \mathfrak{K}_bC_{sg}\Vert y- \alpha \mathbf{1}_{\varepsilon G}\Vert\\\nonumber
		&=& \mathfrak{K}_bC_{sg}\Vert (\alpha+\delta)\mathbf{1}_{F}\Vert.
	\end{eqnarray*}
	Taking limits when $\delta$ goes to $0$,
	\begin{eqnarray}\label{p_1}
	\alpha\Vert\mathbf{1}_{\varepsilon G}\Vert\leq \mathfrak{K}_bC_{sg}\Vert \alpha\mathbf{1}_{F}\Vert\leq \mathfrak{K}_bC_{sg}\Vert \alpha\mathbf{1}_{E}\Vert+ \mathfrak{K}_bC_{sg}c_2\alpha\leq \mathfrak{K}_bC_{sg}\Vert\alpha\mathbf{1}_{E}\Vert+ \mathfrak{K}_bC_{sg}c_2^2\Vert x\Vert.
	\end{eqnarray}
	Now, it is only necessary to estimate $\Vert \alpha \mathbf{1}_E\Vert$. For that, we only have to apply Lemma \ref{lemma1} and then, we obtain that
	\begin{eqnarray}\label{p_2}
	\Vert \alpha\mathbf{1}_{E}\Vert\leq C_{sg}(1+\mathfrak{K}_b)\Vert x\Vert.
	\end{eqnarray}
	Using \eqref{p_1} and \eqref{p_2}, we obtain the result in this case.
	
	\textbf{Case 1.2:} Now, if $w(G)\leq \lim\sup_{n\rightarrow\infty} w_n$, using Proposition \ref{p:find c0}, $$\max_{\vert\varepsilon\vert=1}\Vert\mathbf{1}_{\varepsilon G}\Vert\leq 2C_{sg}\mathfrak{K}_b.$$
	Hence, $$\alpha \Vert\mathbf{1}_{\varepsilon G}\Vert \leq 2C_{sg}\mathfrak{K}_b\alpha \leq 2c_2C_{sg}\mathfrak{K}_b\Vert x\Vert.$$
	
	\item \textbf{Case 2:} If $\sum_n w_n<\infty$ or $\sup_n w_n =\infty$, using Proposition \ref{p:find c0}, $\mathcal B$ is equivalent to the $c_0$-basis and the result is trivial. 
	
\end{proof}

\section{Proof of the main results}\label{proofs}
\begin{proof}[Proof of Theorem \ref{main1}]: First, we prove the item a). Assume that $\mathcal B$ is $C_a$-$w$-almost-greedy and take two sets $A, B\in\mathbb N^{<\infty}$ with $w(A)\leq w(B)$ and two signs $\varepsilon, \varepsilon'$. First, we show that
	\begin{eqnarray}\label{alm1}
	\Vert \mathbf{1}_{A'}\Vert \leq C_a\Vert \mathbf 1_B\Vert,\; \forall A'\subset A.
	\end{eqnarray}

	Define the element $x:=\mathbf{1}_{A'\setminus B}+\mathbf{1}_{A'\cap B}+(1+\delta)\mathbf{1}_{B\setminus A'}$, with $\delta>0$. Since $w(A')\leq w(A)\leq w(B)$, then $w(A'\setminus B)\leq w(B\setminus A')$. Thus,
	$$\Vert \mathbf{1}_{A'}\Vert= \Vert x-\mathcal G_{\vert B'\setminus A\vert}(x)\Vert \leq C_a\Vert x-P_{A'\setminus B}(x)\Vert\leq C_a\Vert \mathbf{1}_{A'\cap B}+(1+\delta)\mathbf{1}_{B\setminus A'}\Vert.$$
	Hence, taking the limits when $\delta$ goes to $0$, we prove \eqref{alm1}. Taking into account that \eqref{alm1} does not change if we apply the estimate for $\lbrace\varepsilon'_n e_n\rbrace_n$ for any $\vert\varepsilon'\vert=1$, we can assume that $\varepsilon' \equiv 1$.

	Now, to conclude the result, we realize that $\mathbf{1}_{\varepsilon A}\in 2S$ for the real case and $\mathbf{1}_{\varepsilon A}\in 4S$ for the complex case, where 
	$$S=\lbrace \sum_{A'\subset A}\vartheta_{A'}\mathbf{1}_{A'} : \sum_{A'\subset A}\vert \vartheta_{A'}\vert\leq 1\rbrace.$$
	Hence, applying this remark in \eqref{alm1} and \cite[Lemma 6.4]{DKO},
	$$\Vert \mathbf{1}_{\varepsilon A}\Vert\leq 2\kappa C_a\Vert \mathbf 1_{\varepsilon' B}\Vert.$$

	This completes the $w$-super-democracy. To show the $w$-disjoint-super-democracy, take $A, B, \varepsilon, \varepsilon'$ as in the beginning with $A\cap B=\emptyset$. Define the element $x:=\mathbf{1}_{\varepsilon A}+(1+\delta)\mathbf{1}_{\varepsilon' B}$ with $\delta>0$. Hence, the set $B$ is the greedy set of $x$ with cardinality $m:=\vert B\vert$. Thus,
	$$\Vert \mathbf{1}_{\varepsilon A}\Vert=\Vert x-\mathcal{G}_m(x)\Vert\leq C_a\Vert x-P_A(x)\Vert = C_a\Vert (1+\delta)\mathbf{1}_{\varepsilon' B}\Vert.$$
	Taking limits when $\delta$ goes to $0$, we obtain that $\mathcal B$ is $C_{sd}$-$w$-disjoint-super-democratic with $C_{sd}\leq C_a$.
	
		The proof of quasi-greediness is trivial since in the definition of $w$-almost-greediness we can take $A=\emptyset$.
	
	b) Assume now that $\mathcal B$ is $C_q$-quasi-greedy and $C_{sd}$-$w$-disjoint-super-democratic. Take $m\in\mathbb N$, $\mathcal G_m(x)=P_A(x)$ and $B$ such that $w(B)\leq w(A)$ and $\Vert x-P_B(x)\Vert<\tilde{\sigma}_{w(A)}^w(x)+\delta$ with $\delta>0$. We have the following decomposition:
	$$x-P_A(x)=P_{(A\cup B)^c}(x-P_B(x))+P_{B\setminus A}(x).$$
	
	On the one hand, since $A\setminus B$ is a greedy set of $x-P_B(x)$,
	\begin{eqnarray}\label{one}
	\Vert P_{(A\cup B)^c}(x-P_B(x))\Vert \leq C_q\Vert x-P_B(x)\Vert.
	\end{eqnarray}
	
	On the other hand, since $w(B\setminus A)\leq w(A\setminus B)$, using Lemma \ref{conv} and $w$-disjoint-super-democracy with $\varepsilon\equiv\lbrace\sgn (e_j^*(x))\rbrace$,
	$$\Vert P_{B\setminus A}(x)\Vert \leq C_{sd}\max_{j\in B\setminus A}\vert e_j^*(x)\vert\Vert\mathbf{1}_{\varepsilon(A\setminus B)}\Vert\leq C_{sd}\min_{j\in A\setminus B}\vert e_j^*(x)\vert\Vert\mathbf{1}_{\varepsilon(A\setminus B)}\Vert$$
	
	Now, by Lemma \ref{propC}, using that $A\setminus B$ is a greedy set of $x-P_B(x)$ and $\min_{j\in A\setminus B}\vert e_j^*(x)\vert = \min_{j\in A\setminus B}\vert e_j^*(x-P_B(x))\vert$,
	\begin{eqnarray}\label{two}
	\Vert P_{B\setminus A}(x)\Vert \leq C_{sd}\min_{j\in A\setminus B}\vert e_j^*(x-P_B(x))\vert\Vert\mathbf{1}_{\varepsilon(A\setminus B)}\Vert\leq 2C_qC_{sd}\Vert x-P_B(x)\Vert.
	\end{eqnarray}
	
	By \eqref{one} and \eqref{two}, we obtain that $\mathcal B$ is $C_a$-$w$-almost-greedy with $C_a\leq C_q+2C_qC_{sd}$.
	
\end{proof}
	
\begin{proof}[Proof of Theorem \ref{main}]:
Assume that $\mathcal B$ is $C_{sg}$-$w$-semi-greedy. To exhibit the $w$-super-democracy and quasi-greediness we consider the different cases that we have considered in the Proposition \ref{proop}.

\item \textbf{Case 1:} $\sum_{n=1}^\infty w_n = \infty$ and $\sup_n w_n <\infty$.

\begin{itemize}
	\item To prove quasi-greediness, take $x\in\mathbb X$ such that  $\vert \supp(x)\vert<\infty$, and without loss of generality we can assume that $\max_j \vert e_j^*(x)\vert \leq 1$, $m\in\mathbb N$ and consider that $w(A_m(x))>\lim\sup_{n\rightarrow\infty} w_n$. Since $\sum_{n}w_n = \infty$, we can choose $E$ and $n_0\in\mathbb N$ with $E> \supp(x)$ and $n_0>\max_{i\in E} i$ such that $$w (E)\leq w(A_m(x))<w(E)+w_{n_0}.$$
	Set $F:=E\cup\lbrace n_0\rbrace$ and $\alpha=\min_{j\in A_m(x)}\vert e_j^*(x)\vert$. Define the element $$y:= (x-\mathcal G_m(x))+(\alpha+\delta)\mathbf{1}_F,$$
	with $\delta>0$. Hence, the greedy set of $y$ is $F$ and then, if the scalars $(a_n)_n$ are given by the TCGA,
	\begin{eqnarray*}
		\Vert x-\mathcal G_m(x)\Vert &\leq&\mathfrak{K}_b\Vert x-\mathcal G_m(x)+\sum_{n\in F}a_n e_n\Vert\leq \mathfrak{K}_bC_{sg} \sigma_{w(F)}^w(y)\\
		&\leq& C_{sg} \mathfrak{K}_b\Vert (x-\mathcal G_m(x))+\sum_{i\in A_m(x)}e_i^*(x)e_i +(\alpha + \delta)\mathbf{1}_F\Vert\\
		&=& C_{sg} \mathfrak{K}_b\Vert x+(\alpha + \delta)\mathbf{1}_F\Vert.
	\end{eqnarray*}
	Taking limits when $\delta$ goes to $0$, 
	\begin{eqnarray}\label{ine1}
	\Vert x-\mathcal G_m(x)\Vert \leq C_{sg} \mathfrak{K}_b(\Vert x\Vert +\Vert \alpha \mathbf{1}_E\Vert +\alpha\Vert e_{n_0}\Vert).
	\end{eqnarray}
	Of course, $\alpha\Vert e_{n_0}\Vert \leq c_2^2\Vert x\Vert$, so we only have to estimate $\Vert \alpha \mathbf{1}_E\Vert$. For that, using Lemma \ref{lemma1}, 
	\begin{eqnarray}
	\alpha\Vert \mathbf{1}_E\Vert \leq C_{sg}(1+\mathfrak{K}_b)\Vert x\Vert.
	\end{eqnarray}

	Then, we have that the basis is quasi-greedy for elements with finite support with 
	$$\Vert x-\mathcal G_m(x)\Vert \leq C_{sg} \mathfrak{K}_b(1+(1+\mathfrak{K}_b)C_{sg} +c_2^2)\Vert x\Vert.$$
	
Define now $C_1= C_{sg} \mathfrak{K}_b(1+(1+\mathfrak{K}_b)C_{sg} +c_2^2)$. To show the quasi-greediness for any $x\in\mathbb X$, we need the following result (see \cite[Lemma 2.2]{Oik}): if $x\in\X$ and $A_m(x)$ is the greedy set of cardinality $m$ of $x$, then for any $\e>0$ there exists $y\in\X$ with $\vert \supp(y)\vert<\infty$ such that $\Vert x-y\Vert<\e$ and $A_m(x)=A_m(y)$. Using that, we proceed as follows:

\begin{eqnarray*}
\Vert x-\mathcal G_m(x)\Vert &\leq& \Vert x-y\Vert + \Vert y-\mathcal G_m(y)\Vert + \Vert \mathcal G_m(y)-\mathcal G_m(x)\Vert\\
&=& \Vert x-y\Vert + \Vert P_{A_m(x)}(x-y)\Vert + C_1\Vert y\Vert\\
&\leq& \Vert x-y\Vert(1+\Vert P_{A_m(x)}\Vert)+C_1\Vert x-y\Vert + C_1\Vert x\Vert\\
&\leq& \Vert x-y\Vert(1+\Vert P_{A_m(x)}\Vert+C_q)+C_1\Vert x\Vert\\
&\leq&\e(1+\Vert P_{A_m(x)}\Vert+C_q)+C_1\Vert x\Vert
\end{eqnarray*}
Taking now limits when $\e$ goes to $0$, we obtain that $\mathcal B$ is $C_q$-quasi-greedy with $C_q\leq C_1$.
	
	Now, consider that $w(A_m(x))\leq \lim\sup_{n\rightarrow\infty} w_n$. Using Proposition \ref{p:find c0}, $$\max_{\vert\varepsilon\vert=1}\Vert\mathbf{1}_{\varepsilon A_m(x)}\Vert\leq 2C_{sg}\mathfrak{K}_b.$$
	
Then, using convexity, $$\Vert\mathcal G_m(x)\Vert \leq \max_{j}\vert e_j^*(x)\vert2C_{sg}\mathfrak{K}_b\leq 2c_2C_{sg}\mathfrak{K}_b\Vert x\Vert.$$
Hence, $\mathcal B$ is quasi-greedy with $C_q \leq 2c_2C_{sg}\mathfrak{K}_b+1$.
	
	\item To show the $w$-super-democracy in this case, take $A, B\in\mathbb N^{<\infty}$ such that $w(A)\leq w(B)$ and two signs $\varepsilon, \varepsilon'$.
If $w(B)>\lim\sup_{n\rightarrow \infty} w_n$, we tan take the set $F$ as before, that is, $F=E\cup\lbrace n_0\rbrace$ such that $w(E)\leq w(B)<w(F)$, $n_0>\max E$ and  $E>A\cup B$. Then, taking the element $x:=\mathbf{1}_{\varepsilon A}+(1+\delta)\mathbf{1}_F$, with $\delta>0$, the greedy set of $x$ is $F$. Using the scalars $(a_i)_{i\in F}$ given by the TCGA, we have that 
	$$\Vert\mathbf{1}_{\varepsilon A}\Vert\leq \mathfrak{K}_b\left\Vert \mathbf{1}_{\varepsilon A}+\sum_{i\in F}a_i e_i\right\Vert \leq \mathfrak{K}_bC_{sg}\sigma_{w(F)}^w(x)\leq \mathfrak{K}_bC_{sg}\Vert (1+\delta)\mathbf{1}_F\Vert.$$
	
	Taking $\delta \searrow 0$, $\Vert \mathbf{1}_{\varepsilon A}\Vert \leq \mathfrak{K}_bC_{sg} (\Vert \mathbf{1}_E\Vert + c_2^2\Vert \mathbf{1}_{\eta B}\Vert)$. Now, as $E>B$, taking the element $x:=(1+\delta)\mathbf{1}_{\varepsilon' B}+\mathbf{1}_E$, using the same ideas that before with $w(B)\geq w(E)$, we obtain that
	$$\Vert \mathbf{1}_E\Vert \leq 2\mathfrak{K}_bC_{sg} \Vert\mathbf{1}_{\varepsilon' B}\Vert.$$
	
	Hence, the basis is $w$-super-democratic with constant $C_{s}\leq 2\mathfrak{K}_b^2C_{sg}^2 + \mathfrak{K}_bC_{sg} c_2^2$.
	
	If $w(B)\leq \lim\sup_{n\rightarrow \infty}w_n$, using Proposition \ref{p:find c0},
	$$\Vert \mathbf{1}_{\varepsilon A}\Vert \leq 2c_2C_{sg}\mathfrak{K}_b\Vert\mathbf{1}_{\varepsilon' B}\Vert.$$
\end{itemize}
		\item 	\textbf{Case 2:} If $\sum_n w_n<\infty$ or $\sup_n w_n =\infty$, using Proposition \ref{p:find c0}, $\mathcal B$ is equivalent to the canonical basis of $c_0$ and the result is trivial.

	The item a) is proved. Now, we show b). Assume that $\mathcal B$ is $C_q$-quasi-greedy and $C_{sd}$-$w$-disjoint-super-democratic. Take $m\in\mathbb N$, $\supp(\mathcal G_m(x))=A_m(x)$ and $z=\sum_{n\in B}a_n e_n$ such that $\Vert x-z\Vert<\sigma_{w(A_m(x))}^w(x)+\delta$ with $\delta>0$.
	If $\alpha = \max_{j\not \in A_m(x)}\vert e_j^*(x)\vert$, we take the element $\nu$ as is defined in \cite{DKK}:
	$$\nu:=\sum_{i\in A_m(x)}T_\alpha(y_i)e_i + P_{(A_m(x))^c}(x) = \sum_{i=1}^\infty T_\alpha(y_i)e_i + \sum_{i\in B\setminus A_m(x)}(e_i^*(x)-T_\alpha(y_i))e_i,$$
	where $y_i=e_i^*(x)-a_i$. Of course, $\nu$ satisfies that $\supp(x-\nu)\subset A_m(x)$ and we will prove that $\Vert \nu\Vert \leq (C_q + 4C_qC_{sd})\Vert x-z\Vert$. 
	
	One the one hand, using Proposition \ref{truncation}, 
	\begin{eqnarray}\label{on1}
	\Vert\sum_{i=1}^\infty T_\alpha(y_i)e_i\Vert \leq C_q\Vert x-z\Vert.
	\end{eqnarray}
	
	On the other hand, since $\vert e_i^*(x)-T_\alpha(y_i)\vert \leq 2\alpha$ for all $i\in B\setminus A_m(x)$, using $C_{sd}$-$w$-disjoint-super-democracy with $\eta\equiv\lbrace\sgn(e_j^*(x-z))\rbrace$, $w(B)\leq w(A_m(x))$ and Lemma \ref{conv},
	\begin{eqnarray}\label{key}
	\Vert \sum_{i\in B\setminus A_m(x)}(e_i^*(x)-T_\alpha(y_i))e_i\Vert &\leq& 2\alpha C_{sd}\Vert\mathbf{1}_{\eta (A_m(x)\setminus B)}\Vert\\\nonumber
	&\leq& 2C_{sd}\min_{i\in A_m(x)\setminus B}\vert e_i^*(x)\vert\Vert\mathbf{1}_{\eta (A_m(x)\setminus B)}\Vert.
	\end{eqnarray}

Now, if we take $\Lambda:=\lbrace n : \vert e_n^*(x-z)\vert \geq \min_{i\in A_m(x)\setminus B}\vert e_i^*(x)\vert, n\not\in A_m(x)\setminus B\rbrace$, $C=\Lambda\cup (A_m(x)\setminus B)$ is a greedy set of $x-z$. Using quasi-greediness,
$$\Vert \mathbf{1}_{\eta(A_m(x)\setminus B)}\Vert \leq C_q\Vert\mathbf{1}_{\eta C}\Vert.$$
Finally, using this fact and Proposition \ref{propC},
\begin{eqnarray*}
2C_{sd}\min_{i\in A_m(x)\setminus B}\vert e_j^*(x)\Vert	\mathbf{1}_{\eta (A_m(x)\setminus B)}\Vert&\leq& 2C_qC_{sd}\min_{i\in C}\vert e_i^*(x-z)\vert\Vert\mathbf{1}_{\eta C}\Vert\\
&\leq& 4C_{sd}C_q^2\Vert x-z\Vert.
\end{eqnarray*}
	
	Hence, the basis $\mathcal B$ is $C_{sg}$-$w$-semi-greedy with $C_{sg}\leq C_q + 4C_{sd}C_q^2$ and b) is finished. Now, we proved the item c). For that, we only have to estimate the inequality \eqref{key} in a different way. Consider that the basis is $C_q$-quasi-greedy and $C_s$-$w$-super-democratic and take the sets $A_m(x), B,C$ and $\eta=\lbrace\sgn (e_j^*(x-z))\rbrace$ as in b). It is clear that $w(B\setminus A_m(x))\leq w(A_m(x)\setminus B)\leq w(C)$, hence, applying the $w$-super-democracy, Lemma \ref{conv} and Proposition \ref{propC} in \eqref{key},
		\begin{eqnarray*}
		\Vert \sum_{i\in B\setminus A_m(x)}(e_i^*(x)-T_\alpha(y_i))e_i\Vert &\leq& 2\alpha C_{s}\Vert\mathbf{1}_{\eta C}\Vert\\
		&\leq& 2C_{s}\min_{i\in A_m(x)\setminus B}\vert e_i^*(x)\vert\Vert\mathbf{1}_{\eta C}\Vert\\
		&=&2C_{s}\min_{i\in C}\vert e_i^*(x-z)\vert\Vert\mathbf{1}_{\eta C}\Vert\\
		&\leq & 4C_{s}C_q\Vert x-z\Vert.
	\end{eqnarray*}

Thus, the basis is $C_{sg}$-$w$-semi-greedy with $C_{sg}\leq C_q+4C_qC_s$.
This completes the proof.
	
\end{proof}

\end{section}

\section{$\rho$-admissibility and semi-greediness}\label{rho}

In the most papers where we study some characterization about greedy-type bases, the more general stage involves only the condition of (strong) Markushevich bases. Then, a natural question is if we can remove the condition of Schauder basis in Theorem \ref{main}. This question is so closed to the Question 1 established in \cite{B}. Here, we present a weaker condition than Schauder to give a characterization of semi-greediness<, that is, the version of Theorem \ref{main} for the weight $w=(1,1,...)$. For that purpose, we consider the following definition that we can find in \cite{BBGHO2}.

\begin{defn}
	For $\rho\geq 1$, we say that $(e_n)_{n=1}^\infty$ is \textbf{$\rho$-admissible} if the following holds: for each finite set $A\subset \mathbb N$, there exists $n_0=n_0(A)$ such that, for all sets $B$ with $\min_{i\in B}i \geq n_0$ and $\vert B\vert\leq\vert A\vert$,
	$$\left\Vert \sum_{n\in A} \alpha_n e_n\right\Vert \leq \rho\left\Vert \sum_{n\in A\cup B} \alpha_n e_n\right\Vert,\; \forall \alpha_n\in\mathbb F.$$
\end{defn}

Of course, this condition is satisfied for Schauder bases, but, in fact, it is satisfied for a more general bases. We remind some classical definitions:
\begin{itemize}
	\item  $(e_n)_{n=1}^\infty$ is \textbf{weakly null} if
	$$\lim_{n\rightarrow \infty} x^*(e_n)=0,\; \forall x^*\in\mathbb X^*.$$
	\item  Given $Y\subset\mathbb X^*$, $(e_n)_{n=1}^\infty$ is \textbf{$Y$-null} if
	$$\lim_{n\rightarrow\infty}y(e_n)=0,\; \forall y\in Y.$$
	\item Given $\kappa\in (0,1]$, a set $Y\subset\mathbb X^*$ is \textbf{$\kappa$-norming} whenever
	$$\kappa\Vert x\Vert \leq \sup_{x^*\in Y, \Vert x^*\Vert\leq 1}\vert x^*(x)\vert,\; \forall x\in\mathbb X,$$
\end{itemize}

In \cite{BBGHO2}, we can find the following result.

\begin{prop}
	Let $\lbrace e_n, e_n^*\rbrace_{n=1}^\infty$ be a biorthogonal system in $\mathbb Xx\mathbb X^*$. Suppose that the sequence $\lbrace \tilde{e}_n := \Vert e_n^*\Vert e_n\rbrace_{n=1}^\infty\subset\mathbb X$ is $Y$-null, for some subset $Y\subset\mathbb X^*$ wich is $\kappa$-norming. Then, $\lbrace e_n\rbrace_{n=1}^\infty$ is $\rho$-admissible for every $\rho>1/k$.
\end{prop}

Some examples of bases that are not Schauder satisfying the above proposition can be found in Section 3 of \cite{BBGHO2}. One of them is the trigonometric system in $C([0,1])$. Our contribution in this case is the following theorem.

\begin{thm}
	Assume that $\mathcal B$ is a $\rho$-admissible basis. Then, $\mathcal B$ is semi-greedy if and only if $\mathcal B$ is quasi-greedy and disjoint-super-democratic.
\end{thm}

\begin{proof}
	We only have to show that semi-greediness implies quasi-greediness and disjoint-super-democracy. The ideas that we use are the same than in Theorem \ref{main} (and \cite[Theorem 1.10]{B}) applying the condition of $\rho$-admissibility.
	
	First, we show super-democracy. Take two sets $A$ and $B$ such that $\vert A\vert\leq \vert B\vert$, $A\cap B=\emptyset$ and two signs $\varepsilon, \eta$. Since the basis is $\rho$-admissible, we can find a set $F$ such that $\vert F\vert=\vert A\cup B\vert$  and $F>A\cup B$. Now, select a set $C\subset F$ such that $\vert C\vert = \vert A\vert$. Hence, 
	\begin{eqnarray}\label{super1}
	\left\Vert \sum_{n\in A\cup B} \alpha_n e_n\right\Vert \leq \rho\left\Vert \sum_{n\in A\cup B\cup C} \alpha_n e_n\right\Vert,\; \forall \alpha_n\in\mathbb F.
	\end{eqnarray}
	
	Consider the element $x:=\mathbf{1}_{\varepsilon A} +(1+\delta)\mathbf{1}_C$, with $\delta>0$. Using the TCGA,
	$$x-\mathcal{CG}_m(x)=\mathbf{1}_{\varepsilon A} + \sum_{i\in C} a_i e_i.$$
	Using semi-greediness and \eqref{super1} with $\alpha_n = \varepsilon_n$ if $n\in A$, $\alpha_n = 0$ if $n\in B$ and $\alpha_n = a_n$ if $n\in C$,
	$$\Vert \mathbf{1}_{\varepsilon A}\Vert \leq \rho \Vert \mathbf{1}_{\varepsilon A} + \sum_{i\in C} a_i e_i\Vert \leq C_{sg}\rho \Vert x-\mathbf{1}_{\varepsilon A}\Vert = C_{sg}\rho\Vert (1+\delta)\mathbf{1}_C\Vert.$$
	
	Taking limits when $\delta$ goes to $0$, we obtain that
	\begin{eqnarray}\label{super2}
	\Vert\mathbf{1}_{\varepsilon A}\Vert \leq C_{sg}\rho\Vert \mathbf{1}_C\Vert.
	\end{eqnarray}
	
	Now, consider $y:=(1+\delta)\mathbf{1}_{\eta B} +\mathbf{1}_C$ with $\delta>0$, 
	$$y-\mathcal{CG}_m(y)=\sum_{i\in B} b_i e_i + \mathbf{1}_C.$$
	
	As before, using semi-greediness and \eqref{super1} with $\alpha_n = 0$ if $n\in A$, $\alpha_n = b_n$ if $n\in B$ and $\alpha_n = 1$ if $n\in C$,
	
	\begin{eqnarray*}
	\Vert \mathbf{1}_{ C}\Vert \leq (1+\rho) \Vert \mathbf{1}_{ C} + \sum_{i\in B} b_i e_i\Vert \leq C_{sg}(1+\rho) \Vert x-\mathbf{1}_{C}\Vert = C_{sg}(1+\rho)\Vert (1+\delta)\mathbf{1}_{\eta B}\Vert.
	\end{eqnarray*}
	
	Taking limits when $\delta$ goes to $0$ and \eqref{super2}, we obtain that
	$$\Vert \mathbf{1}_{\varepsilon A}\Vert \leq C_{sg}^2 \rho(1+\rho)\Vert\mathbf{1}_{\eta B}\Vert.$$
	Thus, the basis is $C_{sd}$-disjoint-super-democratic with $C_{sd}\leq C_{sg}^2\rho(1+\rho)$.
	
	Now, we prove quasi-greediness. Since our basis is strong Markushevich, it is enough to consider $x\in\mathbb X$ with finite support $A=\supp(x)$ as we have said in Theorem \ref{main}. Using the $\rho$-admissibility, we can find a set $C$ such that $\vert C\vert = \vert A\vert$, $C>A$ and
	\begin{eqnarray}\label{semi1}
	\Vert \sum_{n\in A}\alpha_n e_n\Vert \leq \rho\Vert \sum_{n\in A\cup C}\alpha_n e_n\Vert,\; \forall \alpha_n\in\mathbb F.
	\end{eqnarray}
	
	Take $m\in\mathbb N$ and $\delta>0$. Define the element $y:= (x-\mathcal G_m(x))+(\alpha+\delta)\mathbf{1}_F$, where $F\subset C$ with $\vert F\vert=m$, $\alpha = \min_{j\in A_m(x)}\vert e_j^*(x)\vert$ and $A_m(x)$ is the greedy set of $x$ with cardinality $m$. Then, using the TCGA,
	$$y-\mathcal{CG}_m(y) = (x-\mathcal G_m(x))+\sum_{i\in F}a_i e_i.$$
	Using semi-greediness and \eqref{semi1} with $\alpha_n = 0$ if $n\in A_m(x)$, $\alpha_n = e_n^*(x)$ if $n\in A\setminus A_m(x)$, $\alpha_n = a_n$ if $n\in F$ and $\alpha_n = 0$ if $n\in C\setminus F$, 
	$$\Vert x-\mathcal G_m(x)\Vert \leq \rho\Vert y-\mathcal{CG}_m(y)\Vert\leq C_{sg}\sigma_{m}(y) \leq C_{sg}\rho(\Vert x\Vert + \Vert(\alpha+\delta)\mathbf 1_F\Vert).$$
	
	Taking limits when $\delta$ goes to $0$,
	\begin{eqnarray}\label{proof_1}
	\Vert x-\mathcal G_m(x)\Vert \leq C_{sg}\rho(\Vert x\Vert + \Vert \alpha \mathbf{1}_F\Vert).
	\end{eqnarray}
	 Now, take $\eta\equiv\lbrace \sgn(e_i^*(x))\rbrace$ and define $z:=\sum_{i\in A_m(x)}(e_i^*(x)+\delta\eta_i)e_i + P_{(A_m(x))^c}(x)+\alpha \mathbf 1_F$ for $\delta>0$. Thus, by TCGA,
	$$z-\mathcal{CG}_m(z) = \sum_{i\in A_m(x)}b_i e_i + P_{(A_m(x))^c}(x)+\alpha \mathbf{1}_F.$$
	
	Again, using semi-greediness and \eqref{semi1} with $\alpha_n = b_n$ if $n\in A_m(x)$, $\alpha_n = e_n^*(x)$ if $n\in A\setminus A_m(x)$, $\alpha_n = \alpha$ if $n\in F$ and $\alpha_n = 0$ if $n\in C\setminus F$,
	
	\begin{eqnarray*}
		\Vert \alpha \mathbf{1}_F\Vert &\leq& (1+\rho)\Vert  \sum_{i\in A_m(x)}b_i e_i + P_{(A_m(x))^c}(x)+\alpha 1_F\Vert \leq C_{sg}(1+\rho)\Vert z-\alpha \mathbf{1}_F\Vert\\ 
		&=& C_{sg}(1+\rho)\Vert \sum_{i\in A_m(x)}(e_i^*(x)+\delta\eta_i)e_i + P_{(A_m(x))^c}(x)\Vert.
	\end{eqnarray*}
	
	Taking limits when $\delta$ goes to $0$, 
	
	\begin{eqnarray}\label{proof_2}
		\Vert \alpha \mathbf{1}_F\Vert\leq C_{sg}(1+\rho)\Vert x\Vert
	\end{eqnarray}
	
By \eqref{proof_1} and \eqref{proof_2}, $\mathcal B$ is $C_q$-quasi-greedy with $C_q\leq C\rho(1+(1+\rho)C_{sg})$.
\end{proof}


\begin{rem}
	We have studied the characterization of semi-greediness using the $\rho$-admissibility. But, at the moment, we don't know if it is possible to prove the same characterization for $w$-semi-greediness since the condition of the $\rho$-admissibility talks about the cardinality over the sets and not over the weights.
\end{rem}

 \section{Final comments}\label{final}
 In this last section we will discuss two questions. The first one is to show that Remark \ref{rem1} is an improvement respect to the bound of Theorem \ref{walm}. To proved that, we establish the following result that is the weighted version of \cite[Lemma 3.5]{BBG}.
 
 \begin{prop}\label{proprem}
 Assume that $\mathcal B$ is a basis in a Banach space $\mathbb X$. If $\mathcal B$ is $C_{d}$-$w$-democratic and $C_q$-quasi-greedy, then $\mathcal B$ is $C_{s}$-$w$-super-democratic with $C_{s}\leq 4\kappa^2 C_qC_d$, where $\kappa=1$ if $\mathbb F=\mathbb R$ and $\kappa=2$ if $\mathbb F=\mathbb C$.
 \end{prop}
 \begin{proof}
 First, we prove the result for the real case. Consider $A, B$ two sets with $w(A)\leq w(B)$ and two signs $\varepsilon, \eta$. If we denote by $A^{\pm} = \lbrace n\in A : \varepsilon_n = \pm 1\rbrace$, using the democracy with $w(A^{\pm})\leq w(A)\leq w(B)$ and quasi-greediness,
 \begin{eqnarray}\label{demo1}
 \Vert \mathbf{1}_{\varepsilon A}\Vert \leq \Vert \mathbf{1}_{A^+}\Vert + \Vert\mathbf{1}_{A^-}\Vert \leq 2C_d\Vert \mathbf{1}_B\Vert.
 \end{eqnarray}
 Now, we decompose $B$ as the set $A$, that is, $B^{\pm}= \lbrace n\in B : \eta_n = \pm 1\rbrace$. Hence, using quasi-greediness, 
 \begin{eqnarray}\label{demo2}
 \Vert \mathbf{1}_B\Vert\leq \Vert \mathbf{1}_{B^+}\Vert + \Vert\mathbf{1}_{B^-}\Vert\leq 2C_q\Vert \mathbf{1}_{\eta B}\Vert.
 \end{eqnarray}
 Then, by \eqref{demo1} and \eqref{demo2}, the basis is $C_{s}$-$w$-super-democratic with $C_{s}\leq 4C_qC_d$. For the complex case, we can proceed using \cite[Lemma 6.4]{DKO} as in Theorem \ref{main1} to conclude that $\mathcal B$ is $C_{s}$-$w$-super-democratic with $C_{s}\leq 4\kappa^2C_qC_d$.
 \end{proof}

 The second question that we study is related to $w$-super-democracy and $w$-disjoint-super-democracy. We know that, if $w=(1,1,...)$, that is, $w(A)=\vert A\vert$, a basis $\mathcal B$ is super-democratic if and only if $\mathcal B$ is disjoint-super-democratic. Quantitatively,
 \begin{itemize}
 	\item If $\mathcal B$ is $C_s$-super-democratic, then $\mathcal B$ is $C_{sd}$-disjoint-super-democratic with $C_{sd}\leq C_s$.
 	\item If $\mathcal B$ is $C_{sd}$-disjoint-super-democratic, then $\mathcal B$ is $C_s$-super-democratic with $C_s\leq C_{sd}^2$.
 \end{itemize}
This result is trivial. Indeed, if the basis is super-democratic, then it is automatically disjoint-super-democratic. For the converse, if we consider that $\mathcal B$ is $C_{sd}$-disjoint-super-democratic and take $\vert A\vert\leq \vert B\vert$ and $C$ such that $C>(A\cup B)$ with $\vert A\vert=\vert C\vert$,
$$\dfrac{\Vert\mathbf{1}_{\varepsilon A}\Vert}{\Vert \mathbf{1}_{\varepsilon' B}\Vert}=\dfrac{\Vert\mathbf{1}_{\varepsilon A}\Vert}{\Vert \mathbf{1}_{\varepsilon' C}\Vert}\dfrac{\Vert\mathbf{1}_{C}\Vert}{\Vert \mathbf{1}_{\varepsilon' B}\Vert}\leq C_{sd}^2\Rightarrow C_s\leq C_{sd}^2.$$

 Now, we ask the same equivalence for general weights. The result is the following:
\begin{prop}
	Assume that $\mathcal B$ is a basis in a Banach space $\mathbb X$.
	\item[a)] If $\mathcal B$ is $C_s$-$w$-super-democratic, then $\mathcal B$ is $C_{sd}$-$w$-disjoint-super-democratic with $C_{sd}\leq C_s$.
	\item[b)] If $\mathcal B$ is $C_{sd}$-$w$-disjoint-super-democratic, then $\mathcal B$ is $C_s$-super-democratic with $C_s\leq C_{sd}(1+c_2^2C_{sd})$.
\end{prop}

\begin{proof}
Only the item b) requires a proof.	 We proceed as in the proof of Theorem \ref{main}. Take $A$ and $B$ such that $w(A)\leq w(B)$. 

\item \textbf{Case 1:} $\sum_{n=1}^\infty w_n = \infty $ and $\sup_n w_n<\infty$.

\textbf{Case 1.1:} Assume that $\lim\sup_{n\rightarrow \infty}w_n < w(B)$. Since $\sum_n w_n =\infty$, we can take $E$ and $n_0$ such that $n_0>E>A\cup B$ such that
$$w(E)\leq w(B)<w(E\cup \lbrace n_0\rbrace).$$

In this case, since $A\cap (E\cup\lbrace n_0\rbrace)=\emptyset$,
\begin{eqnarray}\label{sup1}
\Vert \mathbf{1}_{\varepsilon A}\Vert \leq C_{sd}\Vert \mathbf{1}_{E\cup \lbrace n_0\rbrace}\Vert \leq C_{sd}\Vert\mathbf{1}_{E}\Vert + C_{sd}c_2\leq C_{sd}(1+c_2^2)\Vert\mathbf{1}_{E}\Vert .
\end{eqnarray}

On the other hand, due to $w(E)\leq w(B)$ and $E\cap B=\emptyset$,
\begin{eqnarray}\label{sup2}
\Vert \mathbf{1}_{E}\Vert \leq C_{sd}\Vert \mathbf{1}_{\varepsilon' B}\Vert.
\end{eqnarray}

Using \eqref{sup1} and \eqref{sup2}, we obtain that $\mathcal B$ is $C_s$-super-democratic with $C_s\leq C_{sd}^2(1+c_2^2)$.

\textbf{Case 1.2}: $w(B)\leq \lim\sup_{n\rightarrow\infty}w_n$. Using the item i) of Proposition \ref{p:find c0}, we obtain that
$$\Vert \mathbf{1}_{\varepsilon A}\Vert\leq C_{sd}c_2^2\Vert\mathbf{1}_{\eta B}\Vert.$$

\item \textbf{Case 2:} If $\sup_n w_n =\infty$, the basis is equivalent to the $c_0$-basis and the result is trivial.

The proof is over.
\end{proof}

\textbf{Question:} Recently, in \cite{BBGHO2}, the authors proved that for Schauder bases and $w=(1,1,...)$, the constants of super-democracy and disjoint-super-democracy are of the same order up to the basis constant, that is, $C_{sd}\leq C_s\leq 2(\mathfrak{K}_b+1)C_{sd}+\kappa \mathfrak{K}_b$, where $\kappa =\sup_{n}\Vert e_n\Vert\Vert e_n^*\Vert$. Is it possible to show the same result for general weights?

%
%

%

\end{document}